%% file: main.tex
\documentclass[11pt]{amsart}

\input{Symbol}

\begin{document}
\title{Generalized property F for weakly group-theoretical modular fusion categories}
\author{Enyu Li and Yuze Ruan}
\address{E. Li, Qiuzhen College, Tsinghua University, Beijing, 100084, China}
\email{ley24@mails.tsinghua.edu.cn}

\address{Y. Ruan, Yau Mathematical Sciences Center and Department of Mathematics, Tsinghua University, Beijing, 100084, China}
\email{yuzeruan@mail.tsinghua.edu.cn}

\begin{abstract}
We prove that every weakly group-theoretical modular fusion category, equipped with its canonical positive spherical structure, has generalized property F: the associated Reshetikhin--Turaev mapping class group representations have finite projective image for all compact oriented surfaces with labeled framed points. This generalizes earlier finite-image results for braid group representations associated with weakly group-theoretical braided fusion categories and for mapping class group representations associated with group-theoretical modular fusion categories and Ising categories. In the weakly group-theoretical setting, it also extends the known finiteness result for symmetric mapping class groups of closed surfaces to the full mapping class groups.
\end{abstract}
\maketitle

\section{Introduction}

The property F conjecture of Naidu and Rowell \cite{NR11} states that a braided fusion category has finite images for all associated braid group representations if and only if it is weakly integral. Green and Nikshych \cite{GN21} proved property F for weakly group-theoretical braided fusion categories, extending earlier finite-image results for braided group-theoretical categories \cite{ERW08} and certain metaplectic categories \cite{RW17, GRR20}.

A modular fusion category is said to have \emph{generalized property $F$} if its associated Reshetikhin--Turaev representations of the label-preserving mapping class groups have finite projective image for all compact oriented surfaces with labeled framed points. For the closed torus, finiteness holds for every modular fusion category \cite{NS10}, whereas the higher-genus case is more subtle. Known finite-image results include those for pointed modular categories \cite{bloomquist2018}, Drinfeld doubles of finite groups \cite{FF20}, their twisted analogues \cite{Gustafson18}, and Ising theories \cite{Wright94, JLLHW20, RR24}. For the broader class of weakly group-theoretical modular fusion categories, finiteness has been established upon restriction to the symmetric mapping class groups of closed surfaces \cite{LR24}. This naturally raises the question of whether finiteness extends to the full mapping class groups, including those of surfaces with labeled framed points.

In this paper, we prove that every weakly group-theoretical modular fusion category, equipped with its canonical positive spherical structure, has generalized property F (Theorem \ref{thm:closed_property_F}). The proof is inspired by the strategy of Green and Nikshych \cite{GN21}, who use fiber products of braided $G$-crossed fusion categories to reduce braid group finiteness to the group-theoretical case. Our proof has two main ingredients.

First, let $G$ be a finite group, let $\CC$ be a $G$-graded spherical fusion category, and let $\Sigma$ be a compact oriented surface with nonempty boundary. The string-net space $\SN_{\CC}(\Sigma)$ with empty boundary values admits a canonical decomposition \cite{TuraevVirelizier2012,CGP2023} (see Lemma \ref{lem:blockdecomposition}):
\[
      \SN_{\CC}(\Sigma)=\bigoplus_{\sigma\in[(\Sigma,\partial\Sigma),(BG,*)]}\SN_{\CC}(\Sigma)_{\sigma}.
\]
whose summands are permuted by the mapping class group.  We prove that this decomposition is compatible with
fiber powers: for every $m\geq 1$, there is a canonical equivariant isomorphism, 
\[
   \SN_{\CC^{\boxtimes_G m}}(\Sigma) \cong \bigoplus_{\sigma\in[(\Sigma,\partial\Sigma),(BG,*)]}  \left(\SN_{\CC}(\Sigma)_\sigma\right)^{\otimes m} .
\]
Under this isomorphism, a mapping class $f$ acts from  $(\SN_{\CC}(\Sigma)_{\sigma})^{\otimes m}$ to $(\SN_{\CC}(\Sigma)_{f\cdot\sigma})^{\otimes m}$ by the $m$-fold tensor power of the corresponding action of $f$ from $\SN_{\CC}(\Sigma)_{\sigma}$ to $\SN_{\CC}(\Sigma)_{f\cdot\sigma}$ (see Theorem \ref{thm:blockwise-tensor-power}). 

Second, we establish a categorical reduction: for every weakly group-theoretical fusion category $\AA$, there exist a finite group $H$ and an $H$-graded fusion category $\DD$, Morita equivalent to $\ZZ(\AA)$, such that the eighth fiber power $\DD^{\boxtimes_H 8}$ is group-theoretical
(Theorem~\ref{thm:D8_gt}). The key categorical input is a construction of an invariant Lagrangian subcategory in the eighth Deligne power of any pointed modular fusion category (Proposition~\ref{prop:stable_Lagrangian}).

These two ingredients reduce the generalized property F of weakly group-theoretical modular category to the group-theoretical case.  The finite-image theorem for twisted Dijkgraaf--Witten theories \cite{Gustafson18}, together with
Morita invariance of Drinfeld centers and the comparison between string-net and Turaev--Viro theories and the Reshetikhin--Turaev theories of their Drinfeld centers, gives finiteness for the string-net representations associated with $\DD^{\boxtimes_H 8}$ on surface with one boundary component. The blockwise tensor-power description then implies finiteness for those associated with $\DD$. When $\AA$ is modular, passing to $\ZZ(\DD)$ and using ribbon equivalences
\[
\ZZ(\DD)\simeq\ZZ(\ZZ(\AA))\simeq\AA\boxtimes\AA^{\mathrm{rev}}\boxtimes\AA^{\mathrm{rev}}\boxtimes\AA
\]
yields the generalized property F for $\AA$ (Theorem \ref{thm:closed_property_F}).

The paper is organized as follows. Section~\ref{sec:prelim} reviews the categorical background, the relevant TQFT constructions, and their mapping class group representations. Section~\ref{sec:string-net} studies string-net spaces associated with graded spherical fusion categories and proves Theorem~\ref{thm:blockwise-tensor-power}. Section~\ref{sec:Categorical manipulation} establishes the categorical reduction of Theorem~\ref{thm:D8_gt}. Finally, Section~\ref{sec:Generalized property F} combines these results to prove the main theorem (Theorem~\ref{thm:closed_property_F}).

\subsection*{Notation and conventions}
We use the abbreviations RT, TV and SN for Reshetikhin--Turaev, Turaev--Viro and string-net, respectively. Throughout, all surfaces are compact, connected and oriented. 

\begingroup
\renewcommand{\arraystretch}{1.15}
\setlength{\LTleft}{0pt}
\setlength{\LTright}{0pt}

\begin{longtable}{@{}
  >{\raggedright\arraybackslash}p{0.34\textwidth}
  >{\raggedright\arraybackslash}
    p{\dimexpr0.66\textwidth-2\tabcolsep\relax}
  @{}}

\toprule
Notation & Meaning \\
\midrule
\endfirsthead

% Header repeated on subsequent pages:
\toprule
Notation & Meaning (continued) \\
\midrule
\endhead

% Bottom rule on every page:
\bottomrule
\endfoot

$\one$, $X^*$
  & Tensor unit and dual of $X$ \\

$\Irr(\CC)$
  & Representatives of the simple objects of $\CC$ \\

\addlinespace
$\CC\boxtimes\DD$ & Deligne tensor product \\

$\CC^{\boxtimes m}$ & $m$-fold Deligne tensor power \\

$\CC_g$, $\CC_e$ & Degree-$g$ and neutral components of a $G$-grading \\

$\CC\boxtimes_G\DD$& Fiber product $\bigoplus_{g\in G}\CC_g\boxtimes\DD_g$ \\

$\CC^{\boxtimes_G m}$ & $m$-th fiber power $\bigoplus_{g\in G}\CC_g^{\boxtimes m}$ \\

\addlinespace
$\CC^G$ & Equivariantization by a specified $G$-action \\

$\CC_G$ & De-equivariantization by a specified Tannakian subcategory equivalent to $\Rep(G)$ \\

\addlinespace
$\ZZ(\CC)$ & Drinfeld center of $\CC$ \\

$\CC^{\mathrm{rev}}$ & $\CC$ with the reversed braiding \\

$\CC(A,q)$ & Pointed braided fusion category associated with $(A,q)$ \\
  
$\Sigma_g,\ \Sigma_g^n$ & A closed genus-$g$ surface, and a genus-$g$ surface with $n$ boundary components, respectively. \\

$\Sigma_{g,n},\ \Sigma_{g,n}^{\mathrm{fr}}$ & A closed genus-$g$ surface with $n$ unframed or framed marked points, respectively. \\

$\operatorname{Mod}(\Sigma)$ & The orientation-preserving mapping class group. \\

$\operatorname{PMod}(\Sigma)$ & The pure subgroup, fixing each marked point individually. \\

$\widehat{\operatorname{Mod}}(\Sigma)$ & The integer central extension used in RT theory. \\

$\operatorname{SN}_{\CC}(\Sigma)$ & The string-net  space. \\

$\operatorname{TV}_{\CC}(\Sigma)$ & The Turaev--Viro state space. \\

$\operatorname{RT}_{\AA}(\Sigma)$
& The Reshetikhin--Turaev state space. \\

$\rho^{\CC}_\Sigma$
& The linear string-net/TV representation of
  $\operatorname{Mod}(\Sigma)$. \\

$\overline{\rho}^{\AA}_\Sigma$
& The projective RT representation
  $\operatorname{Mod}(\Sigma)\to
  \operatorname{PGL}(\operatorname{RT}_{\AA}(\Sigma))$. \\

$\widehat{\rho}^{\AA}_\Sigma$
& The extended linear RT representation
  $\widehat{\operatorname{Mod}}(\Sigma)\to
  \operatorname{GL}(\operatorname{RT}_{\AA}(\Sigma))$. \\

$Z^{\AA}_\Sigma(f)$
& RT operator for a chosen lift of $f$; when $\AA=\ZZ(\CC)$, these operators define a linear representation of $\Mod(\Sigma)$ \\

\end{longtable}
\endgroup

\section*{Acknowledgements}
The authors would like to thank Shuang Ming for helpful discussions. This work was supported by Beijing Municipal Science \& Technology Commission [Z221100002722017 to Y.R.].

Generative AI was used to assist with literature searches, create the table and figures, and proofread the manuscript. The authors take full responsibility for the contents of this paper.

\section{Preliminaries}\label{sec:prelim}
\subsection{Spherical fusion categories}
A \emph{fusion category} over $\BC$ is a semisimple, $\BC$-linear abelian, rigid monoidal category with finite-dimensional Hom-spaces, finitely many isomorphism classes of simple objects, and a simple tensor unit $\one$.  For a fusion category $\CC$,  let $\Irr(\CC)$ denote a set of representatives of its simple objects. The Frobenius–Perron dimension of a simple object $X$, denoted by $\FPdim(X)$, is the Perron-Frobenius eigenvalue of its fusion matrix, and \[
\FPdim(\CC)=\sum_{X\in\Irr(\CC)}\FPdim(X)^2
.\] 
The category $\CC$ is called \emph{integral} if every simple object has integer Frobenius-Perron dimension, and \emph{weakly integral} if $\FPdim(\CC)\in\BZ$. It is \emph{pointed} if all its simple objects are invertible. A \emph{pivotal structure} on $\CC$ is a monoidal natural isomorphism 
\[
j:\operatorname{id}_{\CC}\xrightarrow{\sim}(-)^{**}.\]
It is spherical if the associated left and right categorical traces coincide. A fusion category equipped with such a structure is called a \emph{spherical fusion category}. We write $d(X)$ for the categorical dimension of an object $X$. The global dimension is
\[
\dim(\CC)=\sum_{X\in\operatorname{Irr}(\CC)}d(X)^2.
\] 
A fusion category is \emph{pseudo-unitary} if $\dim(\CC)=\FPdim(\CC)$. Every weakly integral fusion category is pseudo-unitary and admits a unique spherical structure for which all simple objects have positive categorical dimensions \cite{ENO05}. 

The \emph{adjoint subcategory} $\CC_{\mathrm{ad}}$ is the fusion subcategory generated by the simple constituents of $X\otimes X^*$, where $X\in\Irr(\CC)$. Set
\[
\CC^{(0)}=\CC,
\qquad
\CC^{(n+1)}=\bigl(\CC^{(n)}\bigr)_{\mathrm{ad}}.
\]
The category $\CC$ is nilpotent if $\CC^{(n)}\simeq\mathrm{Vec}$ for some $n\geq0$. Equivalently, it is obtained from $\mathrm{Vec}$ by successive faithful extensions by finite groups. Here a $G$-extension of $\DD$ is a fusion category with a faithful grading
\[
\CC=\bigoplus_{g\in G}\CC_g,
\qquad
\CC_e=\DD,
\] 
where $\CC_g\neq0$ and $\CC_g\otimes\CC_h\subseteq\CC_{gh}$ for all $g,h\in G$. We refer to \cite{ENO05, GN08, ENO10} for these notions and their basic properties. 

An \emph{action} of a finite group $G$ on a fusion category $\CC$ consists of tensor autoequivalences $T_g$ and monoidal isomorphisms $\mu_{g,h}:T_g\circ T_h\xrightarrow{\sim}T_{gh}$, satisfying the usual coherence conditions. The \emph{equivariantization} $\CC^G$ is the fusion category of pairs $(X,u)$, where $X\in\CC$ and $u_g:T_g(X)\xrightarrow{\sim}X$ satisfy
\[
u_g\circ T_g(u_h)=u_{gh}\circ\mu_{g,h,X},\qquad g,h\in G,
\]
together with the unit condition. Morphisms $f: (X,u)\to(Y,v)$ satisfy $f\circ u_g=v_g\circ T_g(f)$, and the tensor product is induced by that of $\CC$ and the tensor structures of the $T_g$.

Two fusion categories $\CC$ and $\DD$ are \emph{Morita equivalent} if $\DD$ is tensor equivalent to the category of $\CC$-module endofunctors of an indecomposable semisimple left $\CC$-module category, with tensor product given by reverse composition \cite{Mug031, ENO05}. A fusion category is \emph{group-theoretical} if it is Morita equivalent to a pointed fusion category, and \emph{weakly group-theoretical} if it is Morita equivalent to a nilpotent fusion category \cite{ENO11}.

Every group-theoretical fusion category is integral and weakly group-theoretical, while every weakly group-theoretical fusion category is weakly integral, and thus is pseudo-unitary.  The class of weakly group-theoretical fusion categories is closed under fusion subcategories, Deligne tensor products, and extensions and equivariantizations by finite groups \cite{ENO05,ENO11}. Throughout, we equip such categories with their canonical positive spherical structures. These structures are compatible with the above constructions: they restrict to the canonical structures on fusion subcategories, and agree with the induced structures on Deligne tensor products and equivariantizations. Moreover, tensor equivalences between such categories preserve their canonical spherical structures \cite[Corollary~6.2]{NS07a}.

\subsection{Modular fusion categories}
Let $\CC$ be a braided fusion category with braiding $c$. The \emph{Müger centralizer} of a fusion subcategory $\DD\subseteq\CC$ is the full fusion subcategory$$\DD'
=
\left\{
X\in\CC
\ \middle|\
c_{Y,X}\circ c_{X,Y}
=
\operatorname{id}_{X\otimes Y}
\text{ for every }Y\in\DD
\right\}.$$The subcategory $\CC'$ is the \emph{Müger center}, and $\CC$ is \emph{non-degenerate} if $\CC'\simeq\mathrm{Vec}$ \cite{Mug03S}. A spherical structure on a braided fusion category determines a compatible ribbon structure \cite{BK01,DGNO}. A \emph{modular fusion category} is a braided spherical fusion category whose braiding is non-degenerate. A modular fusion category is weakly group-theoretical if its underlying fusion category is weakly group-theoretical. Thus every non-degenerate weakly group-theoretical braided fusion category becomes modular upon choosing its canonical positive spherical structure. 

A fusion subcategory $\EE\subseteq\CC$ is \emph{Tannakian} if it is braided equivalent to $\Rep(G)$, with its standard symmetric braiding. It is \emph{Lagrangian} if it is Tannakian and $\EE'=\EE$ \cite[Def.~4.57]{DGNO}. Let $R=\operatorname{Fun}(G)$ be the regular algebra in $\EE$. The \emph{de-equivariantization} $\CC_G$ is the fusion category of right $R$-modules in $\CC$, with relative tensor product $\otimes_R$ \cite[Sec.~8.23]{EGNO}. The category $\CC_G$ has a canonical $G$-crossed braided structure: it is $G$-graded, carries an action $T$ of $G$ satisfying $T_g((\CC_G)_h)\subseteq(\CC_G)_{ghg^{-1}}$, and has natural isomorphisms
\[
c^G_{X,Y}:X\otimes_R Y\xrightarrow{\sim}T_g(Y)\otimes_R X,\qquad X\in(\CC_G)_g,\quad Y\in\CC_G,
\]
satisfying the crossed-braiding axioms \cite[Def.~8.24.1]{EGNO}. Equivariantization recovers $\CC$:
\[
(\CC_G)^G\simeq\CC
\]
as braided fusion categories \cite[Thm.~4.44]{DGNO}. The neutral component $(\CC_G)_e$ is braided. If $\CC$ is non-degenerate, then the grading is faithful and $(\CC_G)_e$ is non-degenerate. By Natale's theorem \cite[Thm.~1.1]{Nat18}, every weakly group-theoretical modular category $\CC$ admits a presentation
\begin{equation}\label{eq:Natale}
\CC\simeq\DD^G,
\qquad
\DD_e\simeq\mathcal P\boxtimes\mathcal B,
\end{equation}
as braided fusion categories, where $G$ is a finite group, $\DD$ is a faithfully $G$-graded $G$-crossed braided fusion
category, $\PP$ is non-degenerate and pointed, and $\BB$ is either $\mathrm{Vec}$ or an Ising braided category.

The \emph{Drinfeld center} $\ZZ(\AA)$ of a fusion category $\AA$ consists of pairs $(X,\gamma)$, where $\gamma_Y:X\otimes Y\xrightarrow{\sim}Y\otimes X$ is a half-braiding. If $\AA$ is spherical, then $\ZZ(\AA)$ is modular with the induced spherical structure, characterized by the forgetful functor $F:\ZZ(\AA)\to\AA$ being pivotal. Thus $d_{\ZZ(\AA)}(X,\gamma)=d_{\AA}(X)$. In particular, the canonical positive spherical structure on a weakly integral category induces the canonical positive spherical structure on its center \cite[Sec.~2]{DGNO}. If $\AA$ is weakly group-theoretical, so is $\ZZ(\AA)$ \cite[Prop.~4.1]{ENO11}. Two fusion categories are Morita equivalent if and only if their centers are braided equivalent \cite[Thm.~3.1]{ENO11}. Consequently, a fusion category $\AA$ is group-theoretical if and only if there exist a finite group $G$ and a normalized cocycle $\omega\in Z^3(G,\BC^\times)$ such that
\[
\ZZ(\AA)\simeq\ZZ(\mathrm{Vec}_G^\omega),
\]
as braided fusion categories.

If $\AA$ is modular, there is a canonical braided equivalence
\[
\ZZ(\AA)\simeq\AA\boxtimes\AA^{\mathrm{rev}},
\]
where $\AA^{\mathrm{rev}}$ has the same underlying fusion category and reversed braiding $c^{\mathrm{rev}}_{X,Y}=c_{Y,X}^{-1}$ \cite{Mug032, DGNO}. Tensor equivalences between weakly group-theoretical fusion categories preserve their canonical positive spherical structures \cite[Cor.~6.2]{NS07a}, hence the above braided equivalences are automatically ribbon with respect to the canonical twists.  

A non-degenerate braided fusion category is braided equivalent to the center of a pointed fusion category if and only if it contains a Lagrangian subcategory \cite[Thm.~4.64]{DGNO}. Therefore $\AA$ is group-theoretical if and only if $\ZZ(\AA)$ contains a Lagrangian subcategory.

Finally, pointed braided fusion categories $\CC(A, q)$ are classified by pairs $(A,q)$ as follows; see \cite[Sec.~8.4]{EGNO}. Here $A$ is a finite abelian group and $q:A \to \BC^\times$ is a quadratic form on $A$. Simple objects in $\CC(A, q)$ are in one-to-one correspondence to the elements of $A$. Moreover, $\CC(A, q)$ is a ribbon fusion category whose ribbon structure is determined by the twists $\theta_g=q(g)$ for $g \in A = \Irr(\CC(\AA, q))$. It is modular if and only if $b_q(g_1,g_2)=\frac{q(g_1+g_2)}{q(g_1)q(g_2)}$ is a non-degenerate bicharacter on $A$.

\subsection{Mapping class groups}
Let $\Sigma_g$ be a connected closed oriented surface of genus $g$, and let $P=\{p_1,\cdots p_n\}$ be a set of distinct points. We denote the punctured surface by $\Sigma_{g,n}$. Write $\Mod(\Sigma_{g,n})$ for the group of isotopy classes of orientation-preserving diffeomorphisms, allowing puncture permutations, and $\PMod(\Sigma_{g,n})$ for the subgroup of finite index fixing each puncture individually. Equivalently, punctures may be treated as unframed marked points \cite[Sec.~2.1]{FM12}.

Equip each $p_i$ with a nonzero tangent vector. Denote the resulting surface with $n$ framed points are denoted by $\Sigma^{\mathrm{fr}}_{g,n}$. The corresponding mapping class group and pure mapping class group are denoted by $\Mod(\Sigma^{\mathrm{fr}}_{g,n})$ and $\PMod(\Sigma^{\mathrm{fr}}_{g,n})$ respectively (cf.~\cite{RR24}). For $n=0$, all groups reduce to $\Mod(\Sigma_g)$.

Removing disjoint small disks around the marked points gives a compact surface $\Sigma^n_g$ with $n$ boundary components. Denote by $\operatorname{Mod}(\Sigma^n_g)$ its mapping class group, with representatives and isotopies fixing the boundary pointwise. Capping with disks carrying framed points gives the following identification \cite[Sec.~4.2.5]{FM12}

\begin{equation}\label{eq:capping}
\Mod(\Sigma^n_g)\cong\PMod(\Sigma^{\mathrm{fr}}_{g,n}).    
\end{equation}
 
Boundary twists correspond to Dehn twists about framed points; forgetting the framings quotients out these twists \cite[Prop.~3.19]{FM12}.

% Suppose $n\geq 1$ and $\Sigma_g$ embeds in a closed oriented surface $\Sigma'$. If $\Sigma'$ is not a closed annulus and no component of  $\Sigma'\setminus \Sigma_g$ is an open disk, extension by the identity induces an injection \cite[Thm~3.18]{FM12}
% \[
% \PMod(\Sigma^{\mathrm{fr}}_{g,n})
% \cong\operatorname{Mod}(\Sigma^n_g)
% \hookrightarrow\operatorname{Mod}(\Sigma').
% \]

\subsection{Turaev--Viro TQFT via string nets}\label{subsec:stringnets}

 We recall Turaev--Viro TQFT via the Kitaev--Levin--Wen string-net model \cite{Kitaev03,LevinWen05}, following Kirillov's formulation \cite{Kirillov11}. 
Throughout this article, the bordered string-net and TV spaces carry empty boundary value.  

Let $\CC$ be a spherical fusion category and let $\Sigma$ be a compact connected oriented surface, possibly with boundary. A \emph{colored graph with empty boundary value} in $\Sigma$ is a finite graph with oriented edges embedded in the interior of $\Sigma$. Each edge is labeled by an object of $\CC$, and reversing an edge orientation replaces its label by its dual. Each vertex $v$ is labeled by a morphism
\[
  \varphi_v\in\Hom_{\CC}\!\left(\one,
  X_{h_1}\otimes\cdots\otimes X_{h_n}\right),
\]
where $h_1,\dots,h_n$ are its incident half-edges in counterclockwise order, and $X_{h_i}$ is the edge label, dualized when the edge points toward $v$. A starting half-edge is chosen to write the tensor product; changing it uses the canonical cyclic identifications given by the pivotal structure.

Let $\operatorname{Graph}_{\CC}(\Sigma)$ be the vector space spanned by these colored graphs. The \emph{string-net space} is
\[
  \SN_{\CC}(\Sigma)  :=\operatorname{Graph}_{\CC}(\Sigma)/N_{\CC}(\Sigma),
\]
where $N_{\CC}(\Sigma)$ is spanned by \emph{null combinations}: finite sums $\sum_j c_j\Gamma_j$ whose graphs agree outside an embedded closed disk $D\subset\operatorname{int}(\Sigma)$, have identical transverse boundary data on $\partial D$, and satisfy
\[
  \sum_j c_j\,\langle\Gamma_j\cap D\rangle=0.
\]
Here $\langle\,\cdot\,\rangle$ denotes disk evaluation in the graphical calculus of $\CC$ \cite[Secs.~2--3]{Kirillov11}. 
The orientation-preserving mapping class group $\Mod(\Sigma)$ with homeomorphisms and isotopies fixing $\partial\Sigma$ pointwise acts naturally on the string-net space by  \( \rho_{\Sigma}^\CC(f)([\Gamma])=[f(\Gamma)].\)

\begin{Thm}[Kirillov \cite{Kirillov11}, see also \cite{Bartlett2026StringNets}]\label{Thm:SNTV}
   Let $\Sigma$ be a compact surface possibly with boundary. There is a canonical isomorphism of mapping class group representations:
    \[
    \TV_{\CC}(\Sigma)\cong \SN_{\CC}(\Sigma).
    \]
\end{Thm}

We now show that finiteness of the mapping class group image on string-net spaces of closed surface reduces to the case of the corresponding result of surface with one boundary component, whose string-net space is easier to compute as we shall see later.
Fix $g\geq 0$, let $\Sigma_g$ be the closed oriented surface of genus $g$, $D_1$ a closed disk embedded in $\Sigma_g$, and $\Sigma_{g}^1:=\Sigma_g\setminus \mathrm{int}( D_1)\overset{\iota}{\hookrightarrow} \Sigma_g $. Consider the mapping class groups $\Mod(\Sigma_g)$ and $\Mod(\Sigma^1_{g})$ where homeomorphisms and isotopies fix the boundary pointwise, extending by identity on the disk $D_1$ gives a natural homomorphism  
$
\kappa:\Mod(\Sigma_{g}^1)\to \Mod(\Sigma_g).
$ 
\begin{lem}\label{lem:puncture-filling}
If the representation $\rho_{\Sigma_{g}^1}^{\CC}:\Mod(\Sigma_{g}^1)\to \GL(\SN_{\CC}(\Sigma_{g}^1))$ has finite image, then so does the representation $\rho_{\Sigma_g}^{\CC}:\Mod(\Sigma_g)\to GL(\SN_{\CC}(\Sigma_g))$.
\end{lem}
\begin{proof}
Any two orientation-preserving embeddings of disk in $\Sigma_g$ are ambiently isotopic by disc theorem, so one can isotope a homeomorphism on $\Sigma_{g}$ so that it is the identity on $D_1$, hence $\kappa$ is surjective.  One can also isotope every colored graph in $\Sigma_g$ so that it is disjoint from $D_1$, hence $\iota_*$ is also surjective. Finite image on $\SN_{\CC}(\Sigma_g)$ follows immediately from functorality: 
\[
\begin{tikzcd}[column sep=large, row sep=large]
\Sigma_{g}^1 \arrow[r, hook, "\iota"] \arrow[d, "f"'] & \Sigma_g \arrow[d, "\kappa(f)"] \\
\Sigma_{g}^1 \arrow[r, hook, "\iota"'] & \Sigma_g
\end{tikzcd}
\qquad\Longrightarrow\qquad
\begin{tikzcd}[column sep=large, row sep=large]
\SN_{\CC}(\Sigma_{g}^1) \arrow[r, "\iota_*"] \arrow[d, "\rho_{\Sigma_{g}^1}^\CC(f)"'] 
  & \SN_{\CC}(\Sigma_g) \arrow[d, "\rho_{\Sigma_g}^{\CC}(\kappa(f))"] \\
\SN_{\CC}(\Sigma_{g}^1) \arrow[r, "\iota_*"'] 
  & \SN_{\CC}(\Sigma_g)
\end{tikzcd}
\]
\end{proof}

\subsection{Reshetikhin-Turaev TQFT and its relation to Turaev-Viro TQFT}\label{subsec:rt}
We refer to \cite{RT91, Tur10} for the definition of Reshetikhin-Turaev TQFT and its mapping class group actions. 
Let $\CC$ be a modular fusion category over $\BC$. Put
\[
L=\bigoplus_{V\in\Irr(\CC)}V\otimes V^*,
\qquad
X^+=X,\quad X^-=X^*.
\]
Equip $\Sigma^{\mathrm{fr}}_{g,n}$ with labels $\boldsymbol{X}=(X_1,\ldots,X_n)\in \Obj(\CC)^n$ and signs $\boldsymbol\varepsilon=(\varepsilon_1,\ldots,\varepsilon_n)$.
Its algebraic Reshetikhin-Turaev state space is
\[
\RT_{\CC}
\bigl(\Sigma^{\mathrm{fr}}_{g,n};\boldsymbol{X},\boldsymbol\varepsilon\bigr)
=\operatorname{Hom}_{\CC}\left(\mathbf1,\,X_1^{\varepsilon_1}\otimes\cdots\otimes X_n^{\varepsilon_n}\otimes L^{\otimes g}\right).
\]
The state space of a surface $\Sigma^n_g$ with parametrized boundary circles, together with its homeomorphism
operators, is defined using the corresponding capped surface and capped homeomorphisms \cite[Thm.~V.4.1 and Secs.~V.4.2--V.4.3]{Tur10}. We therefore use the same notation $\RT_{\CC}$ in both descriptions.
Write $\Sigma$ for the resulting decorated surface, suppressing the labels and signs when convenient. $\Mod(\Sigma)$ denotes its label- and sign-preserving framed mapping class group. The mapping-cylinder operator gives a projective representation
\[
\overline\rho^{\CC}_\Sigma: \Mod(\Sigma)\longrightarrow
\operatorname{PGL}\bigl(\RT_{\CC}(\Sigma)\bigr).
\] 
We write $Z^{\CC}_\Sigma(f)$ for a chosen mapping-cylinder operator representing the projective class $\overline\rho^{\CC}_\Sigma(f)$.
The RT construction uses auxiliary Lagrangian subspaces and integer weights to obtain a linear representation $\widehat\rho^{\CC}_\Sigma$ of the integer central extension
\[
1\longrightarrow\mathbb{Z}\longrightarrow\widehat{\Mod}(\Sigma)\longrightarrow\Mod(\Sigma)\longrightarrow 1,
\]
whose central weight subgroup acts by powers of a fixed root of unity. Changing these choices does not change the projective representation up to equivalence \cite[Sec.~IV.6.3]{Tur10}. We suppress them throughout, choosing the Lagrangians compatibly with the gluing under consideration. Mapping classes are not required to preserve these Lagrangians, and identities between RT operators are understood up to a root of unity. When $\AA=\ZZ(\CC)$, equipped with the induced ribbon structure, we use the
standard anomaly-free normalization \cite{TV17}. The central weight then acts
trivially, so
\[
Z^{\ZZ(\CC)}_\Sigma(f)=\widehat{\rho}^{\ZZ(\CC)}_\Sigma(\widehat f)
\]
is independent of the choice of lift $\widehat f\in \widehat{\Mod}(\Sigma)$ and defines a
linear representation
\[Z^{\ZZ(\CC)}_\Sigma:\Mod(\Sigma)\longrightarrow\operatorname{GL}\bigl(\operatorname{RT}_{\ZZ(\CC)}(\Sigma)\bigr).
\]
We retain $\overline{\rho}^{\ZZ(\CC)}_\Sigma$ for its projectivization;
thus
\[
\overline{\rho}^{\ZZ(\CC)}_\Sigma(f)= \bigl[Z^{\ZZ(\CC)}_\Sigma(f)\bigr].
\]

Let $\Sigma'$ be obtained from $\Sigma$ by identifying two boundary components with equal object labels $X$ and opposite signs, using an orientation-reversing identification preserving their base points. With compatible auxiliary Lagrangians, this determines a \emph{gluing homomorphism}
\[
G_X:\RT_{\CC}(\Sigma;X,X^*)\longrightarrow \RT_{\CC}(\Sigma')
\]
\cite[Axiom~V.1.5.4]{Tur10}.
If $f\in \Mod(\Sigma)$ induces $f'\in \Mod(\Sigma')$ after gluing, then up to a root of unity,
\begin{equation}\label{eq:Functorility}
Z^{\CC}_{\Sigma'}(f')G_X
=
G_X Z^{\CC}_{\Sigma}(f).   
\end{equation}
Naturality for weak homeomorphisms, commutativity of gluings along disjoint boundary pairs, and compatibility with disjoint unions are verified in \cite[Sec.~V.4.6]{Tur10}.

The \emph{splitting axiom} states that
\[
G=\bigoplus_{V\in \Irr(\CC)} G_V:
\bigoplus_{V\in \Irr(\CC)}\RT_{\CC}(\Sigma; V,V^*)
\xrightarrow{\sim}\RT_{\CC}(\Sigma'),
\]
is an isomorphism
\cite[Axiom~V.1.6.3, Secs.~V.5.9--V.5.10]{Tur10}.

Finally, disjoint union and RT duality give
\[
\RT_{\CC}(\Sigma_1\sqcup\Sigma_2)\cong \RT_{\CC}(\Sigma_1)\otimes \RT_{\CC}(\Sigma_2), \qquad \RT_{\CC}(\overline{\Sigma})\cong \RT_{\CC}(\Sigma)^*.
\]
Here $\overline{\Sigma}$ reverses the orientation, framed tangent vectors, and signs, while retaining the object labels, see \cite[Sec.~2.1, Eq.~(2.15)]{RR24}.

\iffalse
For a modular fusion category $\AA$, write $\RT_\AA(\Sigma,V)$ for
its state space with parametrized boundary circles labelled by
$V=(V_1,\ldots,V_b)$, where $V_j\in\AA$. We use the projective
mapping class group actions and their standard RT linearizations
on the extended mapping class groups. A framed marked point,
equipped with a nonzero tangent vector and an object label, may
equivalently be replaced by a labelled boundary circle
\cite[Section~2.1]{RR24}.
\fi

\begin{lem}\label{lem:rt-scalar-control}
 Let $\CC$ be modular and $\Sigma$ be a decorated surface. The RT representation $\overline{\rho}_{\Sigma}^\CC:\Mod(\Sigma)\to\mathrm{PGL}(\RT_\CC(\Sigma))$ has finite image if and only if its linear extension $\widehat{\rho}_{\Sigma}^\CC:\widehat{\Mod}(\Sigma)\to\mathrm{GL}(\RT_\CC(\Sigma))$ has finite image. 
\end{lem}
\begin{proof}
Suppose $\overline{\rho}_{\Sigma}^\CC$ has finite image and $\Sigma=\Sigma_{g,n}^{\mathrm{fr}}$ (the case of $\Sigma_{g}^n$ is similar). Assume $\RT_\CC(\Sigma)$ is nonzero. Let $\widehat{\PMod}(\Sigma)$ the preimage of $\PMod(\Sigma)$ in $\widehat{\Mod}(\Sigma)$. Since $\PMod(\Sigma)$ is generated by finitely many Dehn-twists, $\widehat{\rho}_{\Sigma}^\CC\left(\widehat{\PMod}(\Sigma)\right)$ is generated by images of the anomaly factors and lifts of finitely many Dehn twists. These have finite order by \cite{Vafa88} also see \cite[Secs.~2.1--2.2]{RR24} and \cite[Cor.~8.18.2]{EGNO}. Consequently, their determinants generate a finite subgroup of \(\mathbb C^\times\). Since \[\det(\lambda\id)=\lambda^{\dim \RT_\CC(\Sigma)},\]
its subgroup of scalar matrices is finite.  Therefore $\widehat{\rho}_{\Sigma}^\CC\left(\widehat{\PMod}(\Sigma)\right)$ has both finite scalar subgroup and finite projective image, and hence is finite. But \(\widehat{\PMod}(\Sigma)\) has finite index in \(\widehat{\Mod}(\Sigma)\), as \(\PMod(\Sigma)\) has finite index in \(\Mod(\Sigma)\), it follows that $\widehat{\rho}_{\Sigma}^\CC$ has finite image. The converse is trivial.
\end{proof}

By preceding lemma, we say a modular category $\CC$ has \emph{property F} for $\Sigma$ if either $\overline{\rho}_{\Sigma}^\CC$ or $\widehat{\rho}_{\Sigma}^\CC$ has finite image, and we say $\CC$ has \emph{generalized property F} if it has property F for every decorated surface.

\begin{lem}\label{lem:rt-boundary-reduction}
Let $\CC$ be modular, $g\geq0$, and $n\geq1$. If $\CC$ has property F for $\Sigma_{2g+n-1}$, then it also has property F for $\Sigma^{\mathrm{fr}}_{g,n}$ and $\Sigma_{g}^n$.
\end{lem}
\begin{proof}
It suffices to prove finiteness for the representation $\overline\rho^{\CC}_{\Sigma^{n}_{g}}$, since it is identified with restriction of $\overline\rho_{\Sigma^{\mathrm{fr}}_{g,n}}^\CC$ to $\PMod(\Sigma^{\mathrm{fr}}_{g,n})$ by capping and $\PMod(\Sigma^{\mathrm{fr}}_{g,n})$ is a finite index subgroup of $\Mod(\Sigma^{\mathrm{fr}}_{g,n})$. Put $\Sigma=\Sigma_{g}^n$ and let
$D\Sigma=\Sigma\cup_{\partial\Sigma}\overline\Sigma$ be its double. This is a
closed surface of genus $2g+n-1$. The splitting and duality axioms give
\begin{equation}\label{eq:splitting}
\RT_\CC(D\Sigma)\cong\bigoplus_{\boldsymbol{V}\in\Irr(\CC)^n}\RT_\CC(\Sigma,\boldsymbol{V})\otimes\RT_\CC(\Sigma,\boldsymbol{V})^*.
\end{equation}
Each mapping class $f\in \Mod(\Sigma)$ extends by the identity on $\overline\Sigma$ to a class $\iota(f)\in \Mod(D\Sigma)$. Each summand in \eqref{eq:splitting} gives an injection (by choosing a nonzero vector in $\RT_\CC(\Sigma,\boldsymbol{V})^*$)
\[
\iota_{\boldsymbol{V}}:\RT_{\CC}(\Sigma,\boldsymbol{V})\hookrightarrow\RT_\CC(\Sigma,\boldsymbol{V})\otimes\RT_\CC(\Sigma,\boldsymbol{V})^*\hookrightarrow \RT_\CC(D\Sigma).
\]
Naturality of gluing \eqref{eq:Functorility} gives the following equality up to a root of unity 
\[
Z_{D\Sigma}(\iota(f))\iota_{\boldsymbol{V}}=\iota_{\boldsymbol{V}}Z_{\Sigma}(f).
\]
Therefore, finiteness of projective image for $\Mod(\Sigma_{2g+n-1})$ representation implies finiteness of projective image for $\PMod(\Sigma^{\mathrm{fr}}_{g,n})$ with the labels $\boldsymbol{V}\in \Irr(\CC)^n$. Lemma~\ref{lem:rt-scalar-control} gives finite extended linear image. For non-simple boundary labels, additivity gives a finite direct sum of these representations of the same extended group, with trivial
action on the multiplicity spaces. Its image is again finite.
\end{proof}

Finally, we state the theorem relating the two theories. 
\begin{Thm}[Turaev--Virelizier--Balsam\cite{TuraevVirelizier2010,Balsam2010III}]
\label{thm:tvrt}
Let $\CC$ be a spherical fusion category and $I_\CC:\CC\to\ZZ(\CC)$ be the adjoint to the forgetful functor. There are natural isomorphisms of mapping class group representations
\begin{align*}&\TV_\CC(\Sigma_g)\cong\RT_{\ZZ(\CC)}(\Sigma_g),\\&\TV_\CC(\Sigma_g^1)\cong\RT_{\ZZ(\CC)}(\Sigma_g^1;I_\CC(\one)).    
\end{align*}
\end{Thm}
Combining Theorem \ref{Thm:SNTV}, Lemma \ref{lem:puncture-filling}, Lemma \ref{lem:rt-boundary-reduction} and Theorem \ref{thm:tvrt}, we have
\begin{Cor}\label{cor:SNRT}
    Let \(\CC\) be a spherical fusion category. If the representation \[\rho_{\Sigma_{g}^1}^\CC:\Mod(\Sigma_{g}^1)\to \GL(\SN_\CC(\Sigma_{g}^1))\] has finite image for every genus $g$, then $\ZZ(\CC)$ has property F for arbitrary compact surface $\Sigma$. 
\end{Cor}

\section{Canonical decomposition of string-net spaces}\label{sec:string-net}
In this section we compute the string-net space when the category is $G$-graded. We first recall the canonical decomposition of string-net space into blocks permuted by mapping class groups (Lemma \ref{lem:blockdecomposition}), then we express the string-net space of surface with nonempty boundary as a direct sum of Hom-spaces of $\CC$ (Lemma \ref{lem:spine-isomorphism}), which allows us to show that the fiber product of the underlying category induces the blockwise tensor product action of the mapping class group (Theorem \ref{thm:blockwise-tensor-power}). For our purpose we restrict to string-nets with empty boundary value, though the result in this section can be easily extended to the general case. 

Let $G$ be a finite group and $\CC=\bigoplus_{k\in G} \CC_k$ be a $G$-graded spherical fusion category. A $\CC$-colored graph $\Gamma$ in $\Sigma$ is \emph{homogeneous} if each edge label is homogeneous. Note that for each vertex $v$ of a homogeneous graph with half-edge labels $X_{h_1}, \cdots, X_{h_m}$, 
\[
\varphi_v\in \Hom_\CC(1,X_{h_1}\otimes \cdots\otimes X_{h_m})
\]
is nonzero only if $\deg(X_{h_1})\cdots \deg(X_{h_m})=e\in G$. We call such graphs \emph{balanced}.
 
A balanced homogeneous graph $\Gamma$ gives a class in \([(\Sigma,\partial \Sigma),(BG,*)]\) or equivalently an isomorphism class of principal \(G\)-bundles equipped with a trivialization over \(\partial\Sigma\) \cite{TuraevVirelizier2012,CGP2023}. For our purpose, we consider the case when $\partial\Sigma\neq \varnothing$. Fix a set of basepoints $Q=\{q_1,\cdots,q_n\}\subset \partial \Sigma$ at the boundary (one for each component), with $\partial_j$ boundary loops based at $q_j$, then  
\[
[(\Sigma,\partial \Sigma),(BG,*)]
% \cong[\Pi_1\Sigma,\bb G]_V
\cong\HH(\Sigma;G):=\{\sigma:\Pi_1(\Sigma,Q)\to \bb G\mid\sigma(\partial _j)=e\}
\]where $\Pi_1(\Sigma,Q)$ is the fundamental groupoid restricted to $Q$ and $\bb G$ is the one-object groupoid with automorphism group $G$. Thus we can also consider the transverse intersection of $\Gamma$ (or holonomy in the language of principal bundles) with a curve $\gamma$ with endpoints in $Q$: \[
\mathrm{\sigma}_{\Gamma}(\gamma)=\prod_{p\in \Gamma\cap \gamma}(\deg X_{p})^{\epsilon(p)}
\]
 where $X_p$ is the edge label at the crossing $p$, $\epsilon(p)=\pm 1$ is the sign of the intersections and the product is taken in the order in which the curve $\gamma$ meets the graph $\Gamma$. A generic homotopy of the curve is generated by three types of moves: isotopy through transverse paths, tangency and vertex crossing. The first move preserves the items in the product, the second inserts or removes an adjacent pair $g,g^{-1}$, while the third move leaves the product unchanged because the vertex is balanced. As a consequence, \(\sigma_\Gamma\) depends only on the homotopy class of $\gamma$ and we have a well-defined morphism $\mathrm{\sigma}_{\Gamma}:\Pi_1(\Sigma,Q)\to\bb G$. 
   
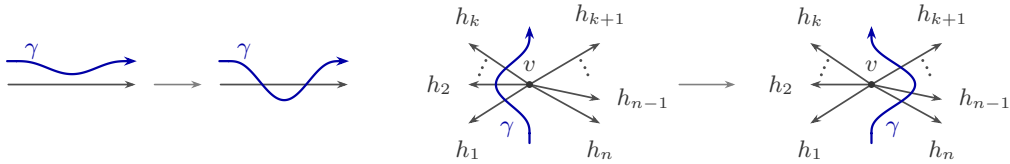
\begin{figure}[htbp]
  \centering \input{moves.tex}
  \label{fig:moves}
\end{figure}

% The index set here is exactly the relative non-abelian Čech cohomology set \(
% \HH(\Sigma;G)\cong\check H^1(\Sigma,\partial\Sigma;\underline{G})\). In particular, $\HH(\Sigma_{g,1};G)=\{\sigma\in\Hom(\pi_1(\Sigma_{g,1},q_1),G):\sigma(\partial_1)=e\}\cong\Hom(\pi_1(\Sigma_g),G)$ as a set.

The mapping class group $\Mod(\Sigma)$ acts naturally on $\HH(\Sigma;G)$ by 
\[
  f\cdot\sigma=\sigma\circ f_*^{-1},f\in\Mod(\Sigma),\sigma\in \HH(\Sigma;G).
\]We have \(f\cdot\sigma_{\Gamma}=\sigma_{f(\Gamma)}\).

\begin{lem}\label{lem:blockdecomposition}
 Let $\Sigma$ be a compact surface with nonempty boundary and $\SN_{\CC}(\Sigma)_{\sigma}$ be the subspace of $\SN_{\CC}(\Sigma)$ spanned by the classes of balanced homogeneous colored graphs $\Gamma$ such that $\mathrm{\sigma}_{\Gamma}=\sigma$. 
Then there is a canonical decomposition:
\[
      \SN_{\CC}(\Sigma)=\bigoplus_{\sigma\in\HH(\Sigma;G)}\SN_{\CC}(\Sigma)_{\sigma}
\]with the action of $f\in\Mod(\Sigma)$ permuting the summands by $\rho_{\Sigma}^\CC(f)(\SN_{\CC}(\Sigma)_{\sigma})=\SN_{\CC}(\Sigma)_{f\cdot\sigma}$.
\end{lem}
\begin{proof}
Every graph is a linear combination of homogeneous graphs, and a homogeneous graph can have a nonzero string-net class only if it is balanced, hence \(\SN_{\CC}(\Sigma)=\sum_{\sigma}\SN_{\CC}(\Sigma)_{\sigma}\). To see directness it suffices to find a set of projectors.
Decomposing every object into homogeneous components gives a canonical map \(\Gamma\mapsto \sum_{l}\Gamma_{l}\), hence for every $\sigma$ there is a canonical map \(\widetilde{\Phi_{\sigma}}:\Gamma\mapsto \sum_{\sigma_{\Gamma_{l}}=\sigma}\Gamma_{l}\).
If $\sum_i c_i \Gamma_i$ is a null combination such that $\Gamma_j$ coincides outside of a disk, decomposing the edge labels that intersect the exterior of the disk into homogeneous components gives $\sum_{i,j}c_i\Gamma_{ij}$, then decomposing the edges in the disk gives $\sum_{i,j,l}c_i\Gamma_{ijl}$. Now note that for fixed $j$, $\sum_{i,l}c_i\Gamma_{ijl}$ is a null combination, also $\sigma_j=\sigma_{\Gamma_{ijl}}$ depends only on the labels outside of the disk, as you can always homotope the representative curve so that it avoids the disk. As a consequence, \[\widetilde{\Phi_{\sigma}} (\sum_i c_i \Gamma_i)=\sum_{j:\sigma_{j}=\sigma}\sum_{i,l} c_i\Gamma_{ijl}\]     is still a sum of null combinations. As a consequence, $\widetilde{\Phi_{\sigma}}$ descends to 
\[
\Phi_{\sigma}: \SN_{\CC}(\Sigma)\to \SN_{\CC}(\Sigma)
\]
that preserves a class of a balanced graph $\Gamma$ if $\sigma_{\Gamma}=\sigma$ and sends it to zero otherwise. Since \(f\cdot\sigma_{\Gamma}=\sigma_{f(\Gamma)}\), the action of  $f\in\Mod(\Sigma)$ sends generators of $\SN_{\CC}(\Sigma)_{\sigma}$ into $\SN_{\CC}(\Sigma)_{f\cdot\sigma}$.
\end{proof}

When the surface has nonempty boundary, there is an explicit expression of the decomposition of the string-net space in terms of Hom-spaces of \(\CC\) \cite{Runkel20}. Let \(\Sigma=\Sigma_{g}^n\) be the compact surface with \(n\geq1\) boundary components, and distinguish its \(n\)-th boundary component. Choose a cutting system consists of loops \(\{a_i,b_i\}_{i=1}^{g}\) based at $q_n$ and arcs \(\{c_j\}_{j=1}^{n-1}\) joining $q_n$ and $q_j$ such that cutting along them produces a disk (see  Figure \ref{fig:spine}).
Choose the dual graph with loops:
\[\{a_i^\vee,b_i^\vee\}_{i=1}^{g}\,\cup\,\{c_j^\vee\}_{j=1}^{n-1}\] labeled by
\[X=\{X_i\}_{i=1}^{g},Y=\{Y_i\}_{i=1}^{g},Z=\{Z_j\}_{j=1}^{n-1}.\]  and one vertex labeled by 
\[\varphi\in \Hom_\CC(\mathbf{1},T(X,Y,Z))\] 
 where \[T(X,Y,Z)= \bigotimes_{i=1}^{g}
(X_i\otimes Y_i\otimes X_i^*\otimes Y_i^*)
\otimes\!
\bigotimes_{j=1}^{n-1}
(Z_j\otimes Z_j^*)
.\]  
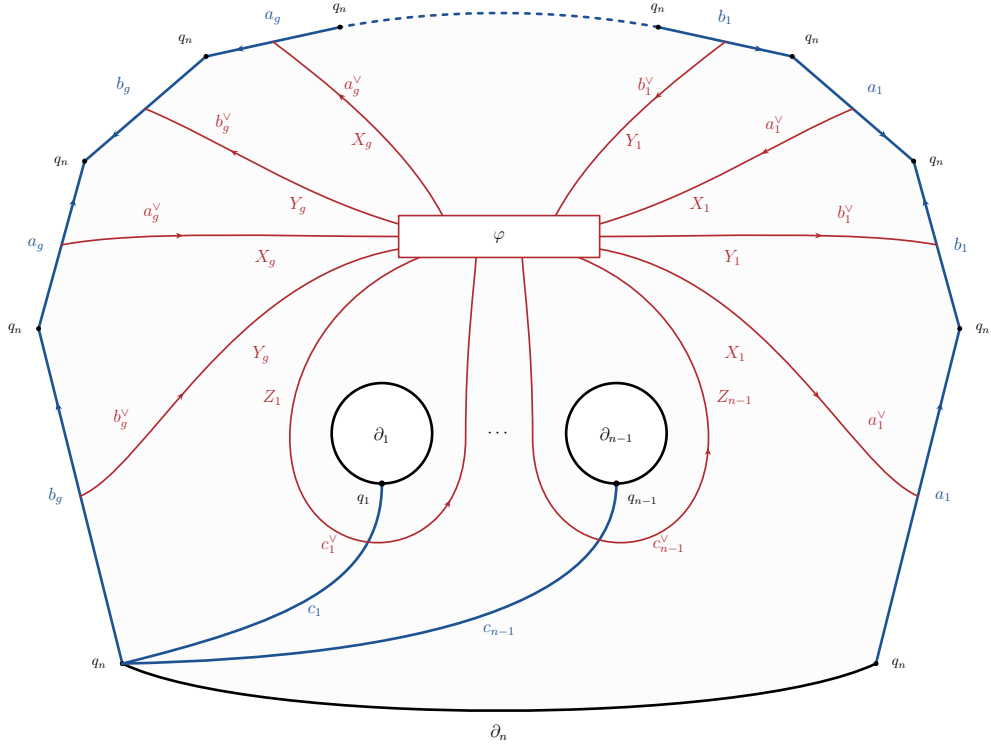
\begin{figure}[htbp]
  \centering
  \resizebox{0.86\linewidth}{!}{\input{surface_polygon_dual_empty.tex}}
  \caption{The graph determined by $\varphi$ dual to the cutting system on $\Sigma_{g}^n$}
  \label{fig:spine}
\end{figure}
 
\begin{lem}[Runkel{ \cite[Thm. 4.1]{Runkel20}}]
\label{lem:spine-isomorphism}For surface $\Sigma$ with nonempty boundary, the above construction gives an isomorphism
\[
  \Psi:\bigoplus_{ \substack{X,Y\in\Irr(\CC)^g\\ Z\in \Irr(\CC)^{n-1}\\ \deg T(X,Y,Z)=e   }}
  \Hom_{\CC}\!\left(\one,T(X,Y,Z)\right)
  \xrightarrow{\sim}\SN_{\CC}(\Sigma).
\]\end{lem}
\begin{proof}We construct the inverse to \(\Psi\) by ``cutting along" the cutting system. Let \(D\) be the disk obtained by cutting \(\Sigma\) along \(a_i,b_i,c_j\).
For each graph, isotope it so that it is transverse to the cutting arcs, with no vertices on them. Along each cutting arc, insert the resolution of the identity
\[
\operatorname{id}_{W}
=
\sum_{S\in\operatorname{Irr}(\CC)}
\sum_{\alpha}
\iota_{S,\alpha}\pi_{S,\alpha},
\]where \(W\) is the ordered tensor product of the labels crossing that arc, with duals determined by the orientations, and
\[
\iota_{S,\alpha}\in\operatorname{Hom}_{\CC}(S,W),
\qquad
\pi_{S,\alpha}\in\operatorname{Hom}_{\CC}(W,S)
\]are dual bases with respect to composition. Cutting and evaluating the resulting graphs in \(D\) defines an element \[\widetilde{\Phi}(\Gamma)\in\bigoplus_{X,Y\in \Irr(\CC)^{g},Z\in \Irr(\CC)^{n-1}} \Hom_{\CC}\!\left(\one,T(X,Y,Z)\right),\] with \(X_i,Y_i,Z_j\) the simple labels at \(a_i,b_i,c_j\), respectively. The sums  are independent of the chosen bases. 
By \cite[Lemma 3.6]{Kirillov11}, this assignment is invariant under passage of a vertex through a cutting arc. Applying the same identity to evaluation and coevaluation morphisms gives invariance under tangencies. Together with isotopy invariance of disk evaluation, this shows that \(\widetilde{\Phi}(\Gamma)\) is independent of the chosen transverse representative.

We next check that \(\widetilde{\Phi}\) annihilates null combinations. Suppose that \(\sum_k c_k\Gamma_k\) is supported in a disk \(D_0\subset\operatorname{int}(\Sigma)\). If \(D_0\) is disjoint from the cutting arcs, it vanishes by definition of the disk evaluation. Otherwise, isotope \(D_0\) into their complement, deforming all \(\Gamma_k\) simultaneously. The preceding invariance reduces this to the first case. Thus \(\widetilde{\Phi}\) induces a map
\(\Phi\) on \(\SN_{\CC}(\Sigma)\). For a graph defined by \(\varphi\) as in Figure \ref{fig:spine}, each cutting arc meets a single edge with a simple label, so cutting and evaluation recover \(\varphi\), hence \(\Phi\circ\Psi=\operatorname{id}.\) Conversely, in \(\Psi\Phi[\Gamma]\), summing the simple-object resolutions restores the identity at every cut, while replacing a disk diagram by its evaluation preserves its string-net class. Therefore \(\Psi\circ\Phi=\operatorname{id}\), proving the claim.
\end{proof}
 Note that the cutting arcs together with the boundary loops generate the fundamental groupoid \(\Pi_1(\Sigma,Q)\), and a graph $\Gamma$ determined by \(\varphi\in \Hom_\CC(\one,T(X,Y,Z))\) gives 
 \[\sigma_\Gamma(a_i)=\deg X_i,\sigma_\Gamma(b_i)=\deg Y_i,\sigma_\Gamma(c_j)=\deg Z_j.\] It follows that the decomposition here refines the canonical decomposition in Lemma \ref{lem:blockdecomposition}, and \(\Psi\) restricts to an isomorphism 
 \[\Psi_\sigma:\bigoplus_{(X,Y,Z)\in J_{\sigma}(\CC)}
  \Hom_{\CC}\left(\one,T(X,Y,Z)\right)
  \xrightarrow{\sim}\SN_{\CC}(\Sigma)_\sigma\] for each \(\sigma\in \HH(\Sigma;G)\), where the index set is   \[ J_{\sigma}(\CC)=
\prod_{i=1}^g\Irr(\CC_{\sigma(a_i)})\times \prod_{i=1}^g \Irr(\CC_{\sigma(b_i)})\times\prod_{j=1}^{n-1} \Irr(\CC_{\sigma(c_j)}).\]

%  \eqref{eq:spine-isomorphism} restricts to\begin{equation}\label{eq:spine-block}
%   \bigoplus_{\substack{X_i,Y_i\in\Irr(\CC)\\
%     \deg X_i=\sigma(a_i^\vee),\ \deg Y_i=\sigma(b_i^\vee)}}
%   \Hom_{\CC}\!\left(\one,W(\mathbf S)\right)
%   \xrightarrow[\sim]{\Psi_{\CC}}
%   \SN_{\CC}(\Sigma)_\sigma.
% \end{equation}

We can now state the main result of this section. For $m\geq1$, let
\[
  \CC^{\boxtimes_G m}
  :=\bigoplus_{k\in G}\CC_k^{\boxtimes m}
  \subseteq\CC^{\boxtimes m}
\]
be the full subcategory with the restricted product spherical structure and the grading
$(\CC^{\boxtimes_G m})_k=\CC_k^{\boxtimes m}$.
Its simple objects are the external products
$S^{(1)}\boxtimes\cdots\boxtimes S^{(m)}$ whose factors are simple
objects of a common degree. 
% Since this subcategory is full,
% \[\label{eq:coupon-factorization}
%   \Hom_{\CC^{\boxtimes_G m}}
%   (\one,W_1\boxtimes\cdots\boxtimes W_m)
%   \cong
%   \bigotimes_{\ell=1}^{m}\Hom_{\CC}(\one,W_\ell)
% \]
% whenever $W_1\boxtimes\cdots\boxtimes W_m$ belongs to
% $\CC^{\boxtimes_G m}$. All structure maps are componentwise, so
% planar evaluation respects this identification.

\begin{Thm}
\label{thm:blockwise-tensor-power}If $\partial \Sigma\neq \varnothing$,
there is a canonical \(\Mod(\Sigma)\)-equivariant isomorphism
\[
  \Theta: \SN_{\CC^{\boxtimes_G m}}(\Sigma) \xrightarrow{\sim} \bigoplus_{\sigma\in\HH(\Sigma;G)}  \left(\SN_{\CC}(\Sigma)_\sigma\right)^{\otimes m} ,
\]
whose restriction $\Theta_\sigma$ identifies the \(\sigma\)-block \(\operatorname{SN}_{\CC^{\boxtimes_Gm}}(\Sigma)_\sigma\) with \( \left(\SN_{\CC}(\Sigma)_\sigma\right)^{\otimes m}\). If $\rho_{\Sigma,\sigma}^{\CC}(f):\SN_\CC({\Sigma})_{\sigma}\to\SN_\CC({\Sigma})_{f\cdot\sigma} $ is restriction of $\rho_{\Sigma}^{\CC}(f)$ to the $\sigma$-block for $f\in \Mod(\Sigma)$, then \[
\Theta_{f\cdot \sigma}\circ\rho_{\Sigma,\sigma}^{\CC^{\boxtimes_G m}}(f)=(\rho_{\Sigma,\sigma}^\CC(f))^{\otimes m}\circ\Theta_{\sigma}.
\]
% Write $\rho_{\CC,\sigma}(f)$ for the map induced by $f\in \Mod(\Sigma)$ from the $\sigma$-block to the $f\cdot\sigma$-block, and similarly for $\CC^{\boxtimes_G m}$. Then, for every $f\in\Mod(\Sigma)$,
% \[\Theta_{f\cdot\sigma}\circ
%   \rho_{\CC^{\boxtimes_G m},\sigma}(f)
%   =
%   \rho_{\CC,\sigma}(f)^{\otimes m}\circ\Theta_\sigma.
% \]
\end{Thm}
\begin{proof}
Fix $\sigma\in\HH(\Sigma;G)$. In the decomposition of Lemma~\ref{lem:spine-isomorphism}, the $\sigma$-block consists of summands with prescribed degrees of $X,Y,Z$. Since the simple objects of \(\CC^{\boxtimes_Gm}\) are external products of simple objects of a common degree, 
the canonical isomorphism
\[
 \Hom_{\CC^{\boxtimes_G m}}\left(\one,\bigboxtimes_{l=1}^m T\bigl(X^{(l)},Y^{(l)},Z^{(l)}\bigr)\right)  \cong  \bigotimes_{l=1}^{m}\Hom_\CC(\one,T\bigl(X^{(l)},Y^{(l)},Z^{(l)}\bigr))
\]
 gives
\begin{align*}    &\bigoplus_{(X,Y,Z)\in J_{\sigma}(\CC^{\boxtimes_G m})}\Hom_{\CC^{\boxtimes_G m}}\left(\one, T\bigl(X,Y,Z\bigr)\right)\\&\cong 
    \bigoplus_{\substack{(X^{(l)},Y^{(l)},Z^{(l)})\in J_{\sigma}(\CC)\\1\leq l\leq m}}\Hom_{\CC^{\boxtimes_G m}}\left(\one,\bigboxtimes_{l=1}^m T\bigl(X^{(l)},Y^{(l)},Z^{(l)}\bigr)\right) \\ &\cong  \bigoplus_{\substack{(X^{(l)},Y^{(l)},Z^{(l)})\in J_{\sigma}(\CC)\\1\leq l\leq m}}\bigotimes_{l=1}^{m}\Hom_\CC\left(\one,T\bigl(X^{(l)},Y^{(l)},Z^{(l)}\bigr)\right)\\  &\cong \left(\bigoplus_{(X,Y,Z)\in J_{\sigma}(\CC)}\Hom_\CC(\one,T\bigl(X,Y,Z\bigr))\right)^{\otimes m}.
\end{align*}
 Applying Lemma~\ref{lem:spine-isomorphism} gives an isomorphism \(\Theta_\sigma: \operatorname{SN}_{\CC^{\boxtimes_Gm}}(\Sigma)_\sigma\to \left(\SN_{\CC}(\Sigma)_\sigma\right)^{\otimes m}\). Set
$\Theta=\bigoplus_{\sigma\in\HH(\Sigma;G)}\Theta_\sigma$.

Let $\Gamma$ be a balanced homogeneous graph with external-product
edge labels and pure-tensor coupons, and let $\Gamma^{(l)}$
denote its $l$-th coordinate graph, the construction in Lemma \ref{lem:spine-isomorphism} gives
\[
\Theta\left([\Gamma]\right)=[\Gamma^{(1)}]\otimes\cdots\otimes[\Gamma^{(m)}].
\]  This implies \[\Theta\left(\rho_{\Sigma}^{\CC^{\boxtimes_Gm}}(f)[\Gamma]\right)=[f\Gamma^{(1)}]\otimes\cdots\otimes[f\Gamma^{(m)}]=
\left(\rho_{\Sigma}^\CC(f)\right)^{\otimes m}\Theta[\Gamma]\] for every
$f\in\Mod(\Sigma)$, which proves equivariance because the classes of such graphs span the string-net space.\end{proof}

\section{Categorical manipulation}\label{sec:Categorical manipulation}
In this section, for a weakly group-theoretical fusion category $\AA$, we will show that there exists a finite group $H$ and an $H$-graded fusion category $\DD$ Morita equivalent to $\ZZ(\AA)$ whose eighth fiber power is group-theoretical. We first focus on the study of pointed modular categories. For every pointed modular category $\CC(A,q)$,  $\CC(A,q)^{\boxtimes 8}$ contains a Lagrangian subcategory \cite{DGNO, DMNO}. The next proposition shows that one can choose it to be invariant under the diagonal action of $\Aut(A,q)$.

\begin{prop}\label{prop:stable_Lagrangian}
For any non-degenerate pointed category $\CC(A,q)$, $\CC(A,q)^{\boxtimes 8}$ contains an $\Aut(A,q)$-stable Lagrangian subcategory, where the action is diagonal.    
\end{prop}
\begin{proof}
Define the quadratic form $q_m$ on $A^{\oplus m}$ by $q_m(g_1,g_2,\dots, g_m)=\prod^m_{i=1}q(g_i)$. Put $N=2|A|$. Then $\gcd(N-1,|A|)=1$. For every $g\in A$, we have $b_q(g,g)=q^2(g)$, and hence $q^{2|A|}(g)=q^N(g)=1$. By Lagrange's four-square theorem, choose integers $a,b,c,d$ such that:
\[
a^2+b^2+c^2+d^2=N-1.
\]
Consider the matrix
\[
F=
\begin{pmatrix}
 a& b& c& d\\
-b& a&-d& c\\
-c& d& a&-b\\
-d&-c& b& a
\end{pmatrix}.
\]
We have
\[
F^TF=FF^T=(a^2+b^2+c^2+d^2)I=(N-1)I.
\]
Therefore, $F$ defines an automorphism of $A^{\oplus 4}$, for $g=(g_1,g_2,g_3,g_4)\in A^{\oplus 4} $, we have:
\begin{align*}
q_4(Fg)&=\prod^4_{i=1}q(g_i)^{\sum_jF^2_{j,i}}\prod_{i<j}b_q(g_i,g_j)^{\sum_kF_{k,i}F_{k,j}}\\
&=(q_4(g))^{N-1}.    
\end{align*}
Let $L=\{(g,Fg)\mid g\in A^{\oplus4}\}$ and let $\LL\subset\CC(A,q)^{\boxtimes 8}$
be the full subcategory generated by these simple objects. Since $q_8(g,Fg)=q_4(g)^N=1$ and $\dim^2(\LL)=\dim(\CC(A,q)^{\boxtimes 8})$, $\LL$ is a Lagrangian subcategory. By construction, it is invariant under the diagonal action of $\Aut(A,q)$.
\end{proof}

When $|A|$ is odd, the preceding construction gives a stronger conclusion. Consider
\[
 L_\pm=\{(g,\pm Fg)\mid g\in A^{\oplus4}\}.
\]
Both are Lagrangian and invariant under the diagonal action of $\Aut(A,q)$. Moreover, if $(g,Fg)=(h,-Fh)$, then $g=h$ and $2Fg=0$. Since both $F$ and multiplication by $2$ are invertible on $A^{\oplus4}$, we obtain $L_+\cap L_-=\{0\}$. Comparing orders we have
\[
 A^{\oplus8}=L_+\oplus L_-.
\]
Thus $q_8$ is hyperbolic \cite[Sec.~4.5]{GNN09}, and we have the following braided equivalence \cite[Example~8.5.4]{EGNO},
\begin{Cor}\label{Cor:untwisted_double}
When $|A|$ is odd,  we have the following braided equivalence
\[
 \CC(A,q)^{\boxtimes8}
 \simeq_{\mathrm{br}}
 \ZZ\left(\operatorname{Vec}_{A^{\oplus4}}\right).
\]   
\end{Cor}

Let $\AA$ be a weakly group-theoretical fusion category. By \cite[Prop.~3.10]{GN21} (cf.~\cite{Nat18}), there exist a braided $G$-crossed fusion category $\CC=\bigoplus_{g\in G} \CC_g$, and $\CC_e\simeq \CC(A,q)$  such that 
\[
Z(\AA)\simeq \CC^{G}.
\]
Set $\DD=\CC\rtimes G$. Then $\DD$ is Morita equivalent to $\CC^G$ \cite[Prop.~3.2]{Nk08} and is graded by $G\rtimes_{\operatorname{Ad}} G$; its $(g,h)$-component is the copy of $\CC_g$ labelled by $h$.

\begin{lem}\label{lem:fiber_crossed_product}
Let $H=G\rtimes_{\operatorname{Ad}} G$, with $(g,h)(g',h')=(ghg'h^{-1},hh')$. We have the following tensor equivalence:
\[
\DD^{\boxtimes_{H} m}\simeq  \CC^{\boxtimes_G m}\rtimes G.
\]
\end{lem}
\begin{proof}
Write an object of $\DD_{(g,h)}$ as $(X_{i,g},h)$, where $X_{i,g}\in \CC_g$.  Define the functor: 
\[
F: \DD^{\boxtimes_{H} m}\longrightarrow  \CC^{\boxtimes_G m}\rtimes G,
\]
component-wise by
\[
F(\boxtimes^{m}_{i=1}(X_{i,g},h))=(\boxtimes^{m}_{i=1}X_{i,g},h).
\]
The functor $F$ is an equivalence. Its compatibility with the tensor products follows from the calculation below; the tensorator and coherence maps are inherited componentwise. 
\begin{align*}
(\boxtimes^{m}_{i=1}X_{i,g_1},h_1)\otimes (\boxtimes^{m}_{i=1}Y_{i,g_2},h_2)  =&\bigl(\boxtimes^{m}_{i=1}X_{i,g_1}\otimes T_{h_1}(\boxtimes^{m}_{i=1}Y_{i,g_2}),h_1h_2\bigr)\\
=&\bigl(\boxtimes^{m}_{i=1}X_{i,g_1}\otimes \boxtimes^{m}_{i=1}T_{h_1}(Y_{i,g_2}),h_1h_2\bigr)\\
=&\bigl(\boxtimes^{m}_{i=1}X_{i,g_1}\otimes T_{h_1}(Y_{i,g_2}),h_1h_2\bigr),
\end{align*}
therefore,
\begin{align*}
F\left((\boxtimes^{m}_{i=1}(X_{i,g_1},h_1))\otimes (\boxtimes^{m}_{i=1}(Y_{i,g_2},h_2))\right)&=F\bigl(\boxtimes^{m}_{i=1}\bigl(X_{i,g_1}\otimes T_{h_1}(Y_{i,g_2}),h_1h_2\bigr)\bigr)\\    
&=(\boxtimes^{m}_{i=1}X_{i,g_1}\otimes T_{h_1}(Y_{i,g_2}),h_1h_2)\\
&=(\boxtimes^{m}_{i=1}X_{i,g_1},h_1)\otimes (\boxtimes^{m}_{i=1}Y_{i,g_2},h_2)
\end{align*}
\end{proof}

\begin{Thm}\label{thm:D8_gt}
With the notation and assumptions above, set $H=G\rtimes_{\operatorname{Ad}}G$. Then the eighth fiber power $\DD^{\boxtimes_H 8}$ is group-theoretical.
\end{Thm}

\begin{proof}
Since being group-theoretical is invariant under Morita equivalence, from Lemma \ref{lem:fiber_crossed_product}, it suffices to prove $(\CC^{\boxtimes_G 8})^G$ is group-theoretical. From Proposition \ref{prop:stable_Lagrangian}, the trivial component has a $G$-stable Lagrangian subcategory, therefore the result follows from \cite[Lem.~6.3]{GN21} (cf. \cite[Thm.~4.64]{DGNO}).   
\end{proof}

\begin{exmp}
Metaplectic categories $SO(N)_2$ ($N=2r+1$) are the unitary modular categories with the same fusion rules as the type $B$ quantum group category at level $2$, see \cite{HNW13,HNW14,ACRW16} for the characterization of inequivalent theories. Let its $\Zn2$ de-equivariantization be $\CC=\operatorname{TY}(A,\chi,\tau)$ \cite{TY98}, where $\tau=\pm |A|^{-1/2}$, and $\chi$ is the bicharacter. Since associators multiply under fiber products,
\[
 \CC_8:=\CC^{\boxtimes_{\Zn2}8}
 \simeq
 \operatorname{TY}(A^{\oplus 8},\chi^{\oplus8},|A|^{-4}).
\]

The complementary Lagrangians $L_\pm\cong A^{\oplus4}$ constructed in Corollary \ref{Cor:untwisted_double} make $\chi^{\oplus8}$ hyperbolic. By \cite[Prop.~4.7]{Bis18}, the equivariantization $\CC_8^{\Zn2}$ has the same positive sign. Hence \cite[Props~5.7--5.8]{Bis18} give
\[
 \CC_8^{\Zn2}\simeq_{\mathrm{br}}\ZZ\left(\operatorname{Vec}_{A^4\rtimes_{-1} \Zn2}\right).
\]
Here $\Zn2$ acts by inversion, and the double is untwisted.
\end{exmp}

\section{Generalized property F}\label{sec:Generalized property F}
In this section, we prove the generalized property F
for weakly group-theoretical modular fusion categories. We start with the following standard lemma.
\begin{lem}\label{lem:FAtoFI}
Let $P$ be a group, and let $\rho_i:  P\to \PGL(H_i),\ {1\leq i\leq n}$ be projective representations on nonzero finite-dimensional vector spaces. If their tensor-product representation has finite projective image, then each $\rho_i(P)$ is finite.
\begin{proof}
The natural map
\[
\prod^n_{i=1}\PGL(H_i)\to\PGL(\bigotimes^{n}_{i=1} H_i) 
\]
is injective.
\end{proof}
\end{lem}
% \begin{proof}
% Denote the dimension of $H_i$ by $m_i$, as $\otimes^{n}_{i=1} \rho_i(\Gamma)$ is a finite group, the kernel of the following map will be a finite-index subgroup of $\Gamma$, we denote it by $\Gamma_0$.
% \[
% \Gamma\xrightarrow[]{\otimes^{n}_{i=1} \rho_i}\otimes^n_{i=1}\GL_{m_i}(\mathbb{C}). 
% \]
% Therefore, for any $f\in \Gamma_0$, there exist $(\lambda_i(f))_{1\leq i\leq n}\in \mathbb{C}^n$, such that $\prod^n_{i=1}\lambda_i(f)=1$ and $\rho_i(f)=\lambda_i(f) I_{n_i}$. Now, since $\Gamma$ has a finite abelianization, denote the order of the abelianization by $l$, then $\det(\rho_i(f))$ will be an $l$-th root of unity, hence $\lambda_i^{ln_i}=1$, $\ker(\rho_i)\cap \Gamma_0$ has finite index in $\Gamma_0$.
% \end{proof}

\begin{Thm}\label{thm:closed_property_F}
For any weakly group-theoretical modular fusion category $\AA$, equipped with its canonical positive spherical structure, $\AA$ has the generalized property F.   
\end{Thm}
\begin{proof}
Let $\DD$ be the category constructed in Section \ref{sec:Categorical manipulation} and $H=G\rtimes_{\operatorname{Ad}}G$, denote the fiber power $\DD^{\boxtimes_H 8}$ by $\BB$ .  For any $g\geq 0$, from \cite{Gustafson18}, the representation 
\[
\rho_{\Sigma_{g}^1}^{\Vect^{\omega}_Q}:\Mod(\Sigma_{g}^1)\to \GL(\SN_{\Vect^{\omega}_Q}(\Sigma_{g}^1)),
\] 
has a finite image for any finite group $Q$, and normalized cocycle $\omega\in Z^3(Q,\BC^\times)$. Therefore, from Corollary \ref{cor:SNRT}, $\ZZ(\Vect^{\omega}_Q)$ has the generalized property F, which implies, from Theorem \ref{thm:D8_gt}, $\ZZ(\BB
)$ has the generalized property F. In particular, the representation 
\[
Z^{\ZZ(\BB)}_{\Sigma^1_{g}}:\Mod(\Sigma^1_{g})\to \GL (\RT_{\ZZ(\BB)}(\Sigma^1_{g}; I_{\BB}(\one)))
\] has a finite image, using Theorem \ref{thm:tvrt}, it is equivalent to the finiteness of the following representation:
\[
 \rho_{\Sigma_{g}^1}^{\BB}:\Mod(\Sigma_{g}^1)\to \GL(\operatorname{SN}_{\BB}(\Sigma_{g}^1)).
 \]
Thus $\ker(\rho_{\Sigma_{g}^1}^{\BB})$ is a finite index subgroup of $\Mod(\Sigma_{g}^1)$. Set $V_{\sigma}=\SN_{\DD}(\Sigma^1_{g})_{\sigma}$, and 
\[
S=\{\sigma\in \HH(\Sigma^1_g; H): V_\sigma\neq0\}.
\]
Write $K=\ker( \rho_{\Sigma_{g}^1}^{\BB})$   For $f\in K$, Theorem \ref{thm:blockwise-tensor-power} implies that  $f$ fixes every nonzero sector $V_\sigma$ and that
\[
\left(\rho^{\DD}_{\Sigma^1_{g},\sigma}(f)\right)^{\otimes 8} =\id_{V^{\otimes 8}_\sigma}\qquad (\sigma\in S).
\] 
Therefore $\rho^{\DD}_{\Sigma_g^1,\sigma}(f)=\lambda_{\sigma}(f)\id_{V_\sigma}$ with $\lambda_{\sigma}(f)^8=1$. Therefore
\[
 \rho^{\DD}_{\Sigma_g^1}(K)\subseteq\left\{\bigoplus_{\sigma\in S}\lambda_{\sigma}\id_{V_\sigma}:\lambda^8_{\sigma}=1\right\},
\]
so the image is finite. Since $K$ has finite index, the image of the following representation
\[
\rho_{\Sigma_{g}^1}^{\DD}:\Mod(\Sigma_{g}^1)\to \GL(\operatorname{SN}_{\DD}(\Sigma_{g}^1)),
\]
is finite. Again using Corollary \ref{cor:SNRT}, $\ZZ(\DD)$ has the generalized property F. 
Finally, ribbon equivalences 
 \[
 \ZZ(\DD)\simeq \ZZ(\ZZ(\AA))\simeq \AA\boxtimes \AA^{rev}\boxtimes\AA^{rev}\boxtimes \AA,
 \]
together with Lemma \ref{lem:FAtoFI} and \ref{lem:rt-scalar-control}, prove $\AA$ has the generalized property F. 
\end{proof}

\section*{Data Availability}
Data sharing is not applicable to this article as no datasets were generated or analysed during the current study.

\section*{Conflict of interest}
The authors have no conflicts of interest to declare that are relevant to the content of this
article.

\bibliographystyle{abbrv}
\bibliography{Reference}
\end{document}

%% file: Symbol.tex
\usepackage[utf8]{inputenc}
\usepackage{fullpage}
\usepackage[colorlinks,linkcolor=blue,citecolor=blue]{hyperref}
\usepackage{graphicx}
\usepackage{amsmath}
\usepackage{amsthm}
\usepackage{amsfonts}
\usepackage{amssymb}
\usepackage{bbm}
\usepackage{mathtools}
\usepackage{extarrows}
\usepackage{array}   % for \newcolumntype macro
\usepackage{booktabs, longtable}
\newcolumntype{L}{>{$}l<{$}}
\usepackage{tikz-cd}

\usepackage{scalerel}
\DeclareMathOperator*{\bigboxtimes}{\vcenter{\hbox{\scaleobj{1.85}{\boxtimes}}}}
\usepackage{pgf,tikz}
\usetikzlibrary{
  knots,
  hobby,
  decorations.pathreplacing,
  shapes.geometric,
  calc
}
\usetikzlibrary{arrows.meta,decorations.markings}

\usepackage{amsmath}
\usetikzlibrary{arrows.meta,decorations.markings,calc}
\usepackage{mathrsfs}
\usetikzlibrary{arrows}

\newtheorem{Thm}{Theorem}[section]

\newtheorem{Cor}[Thm]{Corollary}
\newtheorem{prop}[Thm]{Proposition}
\newtheorem{lem}[Thm]{Lemma}

\theoremstyle{definition}

\newtheorem{exmp}[Thm]{Example}

\theoremstyle{remark}

\newcommand{\one}{\mathbf{1}}

\usepackage{eucal}
\usepackage{graphicx}
\usepackage{tabularx}
\usepackage{multirow}
\usepackage[all,knot]{xy}

\DeclareMathOperator{\Rep}{Rep}

\DeclareMathOperator{\Vect}{Vec}

\DeclareMathOperator{\Mod}{Mod}

\DeclareMathOperator{\PGL}{PGL}

\DeclareMathOperator{\GL}{GL}

\DeclareMathOperator{\PMod}{PMod}

\DeclareMathOperator{\Irr}{Irr}
\DeclareMathOperator{\id}{id}

\DeclareMathOperator{\Obj}{Obj}
\newcommand{\RNum}[1]{\uppercase\expandafter{\romannumeral #1\relax}}

\newcommand{\BZ}{\mathbb{Z}}
\newcommand{\Zn}[1]{\BZ/{#1}\BZ}
\DeclareMathOperator{\Aut}{Aut}

\usepackage{graphicx}
\graphicspath{ {./graph/} }

\newcommand{\PP}{\mathcal{P}}
\newcommand{\BB}{\mathcal{B}}
\newcommand{\BC}{\mathbb{C}}

\renewcommand{\AA}{\mathcal{A}}
\newcommand{\CC}{{\mathcal{C}}}
\newcommand{\DD}{\mathcal{D}}

\newcommand{\HH}{\mathcal{H}}
\newcommand{\LL}{\mathcal{L}}
\newcommand{\ZZ}{\mathcal{Z}}
\newcommand{\EE}{\mathcal{E}}

\newcommand{\Hom}{\operatorname{Hom}}

\providecommand{\TV}{\operatorname{TV}}
\providecommand{\RT}{\operatorname{RT}}
\newcommand{\SN}{\mathrm{SN}}

\newcommand{\FPdim}{\operatorname{FPdim}}
\newcommand{\bb}{\mathbf{B}}

\def\HH{\mathcal{H}}

%% file: moves.tex
\begin{tikzpicture}[
  x=0.78cm,y=0.86cm,font=\footnotesize,
  line cap=round,line join=round,
  gmGraph/.style={draw=black!70,line width=0.65pt,
    -{Stealth[length=1.45mm,width=1mm]}},
  gmPath/.style={draw=blue!65!black,line width=0.85pt,
    -{Stealth[length=1.7mm,width=1.2mm]}},
  gmTime/.style={draw=black!50,line width=0.6pt,
    -{Stealth[length=1.4mm,width=1mm]}},
  gmLabel/.style={font=\footnotesize,text=black!80,inner sep=1pt},
  gmCurveLabel/.style={font=\footnotesize,text=blue!65!black,inner sep=1pt}
]

% Column 1: crossing an edge tangency, with identical horizontal ends.
% Cubed bumps have zero first and second derivatives at the joins.
\foreach \xx/\dip in {0/0.20,3.60/0.60} {
  \begin{scope}[shift={(\xx,0)}]
    \draw[gmGraph] (-1.08,0)--(1.08,0);
    \node[gmLabel] at (0.96,-0.22) {};
    \draw[gmPath] (-1.10,0.36)--(-0.94,0.36)
      plot[domain=-0.94:0.94,samples=81]
        (\x,{0.36-\dip*(1-(\x/0.94)^2)^3})
      --(1.10,0.36);
    \node[gmCurveLabel] at (-0.68,0.52) {$\gamma$};
  \end{scope}
}
\draw[gmTime] (1.40,0)--(2.20,0);

% Column 2: crossing an n-valent vertex v, with multiple crossings on each side.
% Left fan: h_1,...,h_k; right fan: h_{k+1},...,h_n.
% Both paths agree on their terminal segments and have vertical end tangents.
\foreach \xx/\bend/\labelx/\labely in {8.00/-0.57/-0.62/-0.68,13.80/0.74/0.12/-0.68} {
  \begin{scope}[shift={(\xx,0)}]
    \draw[gmGraph] (-0.24,0)--(-1.28,-0.60);
    \draw[gmGraph] (-0.24,0)--(-1.28,0);
    \draw[gmGraph] (-0.24,0)--(-1.28,0.64);
    \draw[gmGraph] (-0.24,0)--(0.96,-0.60);
    \draw[gmGraph] (-0.24,0)--(0.96,-0.23);
    \draw[gmGraph] (-0.24,0)--(0.96,0.64);
    \node[gmLabel,below=5pt] at (-1.28,-0.60) {$h_1$};
    \node[gmLabel,left=5pt] at (-1.28,0) {$h_2$};
    \foreach \dotx/\doty in {-1.09/0.15,-1.04/0.26,-0.98/0.37}
      \fill[black!70] (\dotx,\doty) circle[radius=0.55pt];
    \node[gmLabel,above=5pt] at (-1.28,0.64) {$h_k$};
    \node[gmLabel,below=5pt] at (0.96,-0.60) {$h_{n}$};
    \node[gmLabel,right=5pt] at (0.96,-0.23) {$h_{n-1}$};
    \foreach \dotx/\doty in {0.75/0.15,0.70/0.26,0.63/0.37}
      \fill[black!70] (\dotx,\doty) circle[radius=0.55pt];
    \node[gmLabel,above=5pt] at (0.96,0.64) {$h_{k+1}$};
    \draw[gmPath] (-0.24,-0.90)--(-0.24,-0.75)
      plot[domain=-0.75:0.75,samples=81,variable=\t]
        ({-0.24+\bend*(1-(\t/0.75)^2)^3},\t)
      --(-0.24,0.90);
    \node[gmCurveLabel] at (\labelx,\labely) {$\gamma$};
    \fill[black!80] (-0.24,0) circle[radius=1.3pt];
    \node[gmLabel,above=4pt] at (-0.24,0) {$v$};
  \end{scope}
}
\draw[gmTime] (10.30,0)--(11.25,0);
\end{tikzpicture}

%% file: surface_polygon_dual_empty.tex
% Input-ready TikZ picture. Load \usepackage{tikz} in the main preamble.
% Empty boundary value: no marked p-points and no boundary-value halfedges.
% Both boundary cuts start at the same left polygon copy of q_n.
\begingroup
\usetikzlibrary{arrows.meta,decorations.markings,calc}
\definecolor{cutblue}{RGB}{30,83,147}
\definecolor{dualred}{RGB}{178,42,47}
\begin{tikzpicture}[font=\large,line cap=round,line join=round,
 primal/.style={cutblue,line width=1.8pt,postaction={decorate},decoration={markings,mark=at position .78 with {\arrow{Stealth[length=5.4pt,width=4.5pt]}}}},
 dual/.style={dualred,line width=1.15pt},
 outflow/.style={postaction={decorate},decoration={markings,mark=at position .66 with {\arrow{Stealth[length=5pt,width=4pt]}}}},
 inflow/.style={postaction={decorate},decoration={markings,mark=at position .66 with {\arrow{Stealth[reversed,length=5pt,width=4pt]}}}},
 duallabel/.style={dualred,font=\large}]
% Mirror-symmetric convex polygon, with raised commutator blocks.
\coordinate (R0) at (9,-6.5);
\coordinate (R1) at (11,1.5);
\coordinate (R2) at (9.9,5.5);
\coordinate (R3) at (7,8);
\coordinate (R4) at (3.8,8.7);
\coordinate (L0) at (-3.8,8.7);
\coordinate (L1) at (-7,8);
\coordinate (L2) at (-9.9,5.5);
\coordinate (L3) at (-11,1.5);
\coordinate (L4) at (-9,-6.5);
\fill[gray!2] (L0)--(L1)--(L2)--(L3)--(L4)
 .. controls (-5.8,-8) and (5.8,-8) .. (R0)--(R1)--(R2)--(R3)--(R4)
 .. controls (1.4,9.15) and (-1.4,9.15) .. (L0)--cycle;
\draw[cutblue,line width=1.8pt,dash pattern=on 3pt off 5pt] (R4)
 .. controls (1.4,9.15) and (-1.4,9.15) .. (L0);
\foreach \k in {0,...,3}{
 \pgfmathtruncatemacro{\next}{\k+1}
 \coordinate (TR\k) at ($(R\k)!.5!(R\next)$);
 \coordinate (TL\k) at ($(L\k)!.5!(L\next)$);
 \ifnum\k<2
  \draw[primal] (R\k)--(R\next);
  \draw[primal] (L\k)--(L\next);
 \else
  \draw[primal] (R\next)--(R\k);
  \draw[primal] (L\next)--(L\k);
 \fi
 \ifodd\k\def\letter{b}\else\def\letter{a}\fi
 \ifnum\k=3
  \node[cutblue,above=8pt] at (TR\k) {$\letter_1$};
 \else\ifnum\k=2
  \node[cutblue,above right=6pt] at (TR\k) {$\letter_1$};
 \else
  \node[cutblue,right=8pt] at (TR\k) {$\letter_1$};
 \fi\fi
 \ifnum\k=0
  \node[cutblue,above=8pt] at (TL\k) {$\letter_g$};
 \else\ifnum\k=1
  \node[cutblue,above left=6pt] at (TL\k) {$\letter_g$};
 \else
  \node[cutblue,left=8pt] at (TL\k) {$\letter_g$};
 \fi\fi
}
\draw[black,line width=1.8pt] (L4)
 .. controls (-5.8,-8) and (5.8,-8) .. (R0);
\node at (0,-8.15) {$\partial_n$};
\foreach \k in {0,...,4}{
 \fill (R\k) circle (1.7pt);
 \fill (L\k) circle (1.7pt);
}
\foreach \k in {0,1,2}{\node[font=\small,right=7pt] at (R\k) {$q_n$};}
\foreach \k in {2,3,4}{\node[font=\small,left=7pt] at (L\k) {$q_n$};}
\node[font=\small,above=6pt] at (R4) {$q_n$};
\node[font=\small,above=6pt] at (L0) {$q_n$};
\node[font=\small,above right=5pt] at (R3) {$q_n$};
\node[font=\small,above left=5pt] at (L1) {$q_n$};
% Boundary components retain the cutting-system basepoints only.
\foreach \i/\x/\boundarylabel in {1/-2.8/1,2/2.8/{n-1}}{
 \filldraw[white,draw=black,line width=1.8pt] (\x,-1) circle (1.20);
 \node at (\x,-1.05) {$\partial_{\boundarylabel}$};
 \coordinate (Q\i) at (\x,-2.20);
}
\node at (0,-1) {$\cdots$};
% Both cuts start at the left copy of q_n; c_{n-1} passes below the first loop.
\draw[cutblue,line width=1.8pt] (L4)
 .. controls (-5.2,-5.55) and (-2.8,-4.65) .. (Q1);
\draw[cutblue,line width=1.8pt] (L4)
 .. controls (-1.7,-6.35) and (2.8,-5) .. (Q2);
\node[cutblue] at (-4.4,-5.23) {$c_1$};
\node[cutblue] at (0,-5.7) {$c_{n-1}$};
\fill (Q1) circle (2.1pt);
\fill (Q2) circle (2.1pt);
\node[below left=4pt,font=\small] at (Q1) {$q_1$};
\node[below right=4pt,font=\small] at (Q2) {$q_{n-1}$};
% Handle dual halfedges, in reflected fans on either side of the coupon.
\draw[dual,inflow] (1.35,4.20)
 .. controls (2.3,6) and (4.35,7.55)
 .. node[pos=.35,right=6pt] {$Y_1$} node[pos=.66,left=6pt] {$b_1^\vee$} (TR3);
\draw[dual,inflow] (2.4,4)
 .. controls (4.8,4.7) and (6.65,6)
 .. node[pos=.36,below=4pt] {$X_1$} node[pos=.66,above=5pt] {$a_1^\vee$} (TR2);
\draw[dual,outflow] (2.4,3.7)
 .. controls (5.25,3.7) and (8.6,3.85)
 .. node[pos=.36,below=6pt] {$Y_1$} node[pos=.68,above=5pt] {$b_1^\vee$} (TR1);
\draw[dual,outflow] (2.4,3.4)
 .. controls (6.4,2.7) and (8.4,-1.8)
 .. node[pos=.38,below left=4pt] {$X_1$} node[pos=.72,above right=4pt] {$a_1^\vee$} (TR0);
\draw[dual,outflow] (-1.35,4.20)
 .. controls (-2.3,6) and (-4.35,7.55)
 .. node[pos=.35,left=6pt] {$X_g$} node[pos=.66,right=6pt] {$a_g^\vee$} (TL0);
\draw[dual,outflow] (-2.4,4)
 .. controls (-4.8,4.7) and (-6.65,6)
 .. node[pos=.36,below=4pt] {$Y_g$} node[pos=.66,above=5pt] {$b_g^\vee$} (TL1);
\draw[dual,inflow] (-2.4,3.7)
 .. controls (-5.25,3.7) and (-8.6,3.85)
 .. node[pos=.36,below=6pt] {$X_g$} node[pos=.68,above=5pt] {$a_g^\vee$} (TL2);
\draw[dual,inflow] (-2.4,3.4)
 .. controls (-6.4,2.7) and (-8.4,-1.8)
 .. node[pos=.38,below right=4pt] {$Y_g$} node[pos=.72,above left=4pt] {$b_g^\vee$} (TL3);
% Peripheral dual loops: each crosses its corresponding blue cut once.
\draw[dual,outflow] (-1.9,3.2)
 .. controls (-4,2.35) and (-5,.5) .. (-5,-1.1)
 .. controls (-5,-4.45) and (-.8,-4.45) .. (-.8,-1.1)
 .. controls (-.8,1.05) and (-.6,2.45) .. (-.55,3.2);
\draw[dual,outflow] (.55,3.2)
 .. controls (.6,2.45) and (.8,1.05) .. (.8,-1.1)
 .. controls (.8,-4.45) and (5,-4.45) .. (5,-1.1)
 .. controls (5,.5) and (4,2.35) .. (1.9,3.2);
\node[duallabel,left=3pt] at (-4.95,-.15) {$Z_1$};
\node[duallabel,right=3pt] at (4.95,-.15) {$Z_{n-1}$};
\node[duallabel] at (-4.05,-3.65) {$c_1^\vee$};
\node[duallabel] at (4.05,-3.65) {$c_{n-1}^\vee$};
% The coupon is the only vertex; no red strand ends on a true boundary.
\node[draw=dualred,line width=1.1pt,fill=white,
 minimum width=4.8cm,minimum height=1cm,inner sep=0pt]
 at (0,3.7) {$\varphi$};
\end{tikzpicture}
\endgroup